\documentclass[12pt]{article}
\usepackage[utf8]{inputenc}
\usepackage{amsthm}
\usepackage{amsmath}
\usepackage{bbm}
\usepackage{amsfonts}
\usepackage{amssymb}
\usepackage[left=2cm,right=2cm,top=2cm,bottom=2cm]{geometry}

\usepackage{hyperref}

\usepackage{tabto}
\TabPositions{2 cm, 4 cm, 6 cm, 8 cm}

\usepackage{tikz-cd}

\usepackage{tikz}
\usetikzlibrary{arrows, automata, positioning}

\usepackage{tocloft}
\usepackage{appendix}

\usepackage{enumerate}

\usepackage{hyperref}
 
\newtheorem{theorem}{Theorem}[section]
\newtheorem{corollary}[theorem]{Corollary}
\newtheorem{lemma}[theorem]{Lemma}

\newtheorem{proposition}[theorem]{Proposition}

\theoremstyle{definition}
\newtheorem{definition}[theorem]{Definition}
\newtheorem{example}[theorem]{Example}
\newtheorem{remark}[theorem]{Remark}

\newtheorem{question}[theorem]{Question}

\newtheorem*{lemma*}{Lemma}
\newtheorem*{proposition*}{Proposition}
\newtheorem*{theorem*}{Theorem}
\newtheorem*{corollary*}{Corollary}

\newcommand{\Hom}{\operatorname{Hom}}

\newcommand{\Ext}{\operatorname{Ext}}

\newcommand{\id}{\operatorname{id}}
\newcommand{\res}{\operatorname{res}}

\newcommand{\HH}{\operatorname{H}}

\begin{document}

\title{No quasi-isometric rigidity for \\ proper actions on CAT(0) cube complexes}
\author{Francesco Fournier-Facio and Anthony Genevois}
\date{\today}
\maketitle

\begin{abstract}
We exhibit a variety of groups that act properly and even cocompactly on median graphs (a.k.a. one-skeletons of CAT(0) cube complexes), with quasi-isometric groups that do not admit any proper action on a median graph. This answes a question of Niblo, Sageev and Wise.
Our examples are all quasi-isometrically trivial central extensions of certain cubulated groups.
\end{abstract}


\section{Introduction}

The study of group actions on median graphs (a.k.a. one-skeletons of CAT(0) cube complexes) is a major subject in modern geometric group theory. This is mainly due to the numerous groups of interest acting on median graphs (such as Artin groups, Coxeter groups, small cancellation groups, random groups, one-relator groups with torsion, $3$-manifold groups, free-by-cyclic groups, graph braid groups, Thompson-like groups, Cremona groups) and to the valuable information that can be deduced from such actions (including Tits alternative, Hilbertian geometry, bi-automaticity, asymptotic dimension, subgroup separability, subgroup distortion, behaviours of negative curvature). Applications include many branches of mathematics, most famously low-dimensional topology with the proof the virtual Haken conjecture.

As is common in the field, it is natural to wonder whether the existence of such actions is a quasi-isometry invariant. During the 2007 AIM Meeting \emph{Problems in geometric group theory}, Niblo, Sageev and Wise asked the following question:

\begin{question}[{\cite[Question 12]{questions}}]
\label{q}

Is the property of acting (properly, cocompactly) on a median graph a quasi-isometry invariant?
\end{question}

In \cite{questions}, the authors also remark that the property of acting without fixed points on a median graph is not a quasi-isometry invariant, since there are groups with property (T) that are quasi-isometric to groups surjecting onto $\mathbb{Z}$ \cite[Theorem~3.6.5]{MR2415834}.

They suggest that a negative answer could come from understanding actions of certain lattices on CAT(0) cube complexes. In particular, if $\Gamma \leq \mathrm{SL}_2(\mathbb{R}) \times \mathrm{SL}_2(\mathbb{R})$ is an irreducible uniform lattice, then it is quasi-isometric to the product of two surface groups, which is cocompactly cubulated (see Example \ref{ex:surface}), but such lattices $\Gamma$ are conjectured to have Property FW \cite{FW}. This is still open, but on the other hand such lattices are known to have the fixed point property for actions on \emph{finite-dimensional} CAT(0) cube complexes \cite{CFI}, and in particular cannot act properly \emph{and} cocompactly on a median graph. 

It is also shown in \cite{wise} (and investigated further in \cite{hagen}) that admitting a geometric action on a median graph is not even preserved under commensurability. Namely, the $(3,3,3)$-triangle group (which coincides with the symmetry group of the tessellation of the Euclidean plane by equilateral triangles) does not act geometrically on a median graph, even though it contains a free abelian subgroup of finite index. As another source of interesting examples, it is worth mentioning that, as a consequence of \cite{graphmanifolds2} and \cite{graphmanifolds}, there exist quasi-isometric fundamental groups of graph manifolds such that one admits a proper and cocompact action on a median graph while the other does not. 

In the opposite direction, acting properly on a median graph is preserved under commensurability, since, by a classical argument, if a group $G$ contains a finite-index subgroup $H$ acting (properly) on some graph $X$ then $G$ naturally acts (properly) on $X^{[G:H]}$.\\

The above discussion shows that the property of acting properly and cocompactly on median graphs is not a quasi-isometry invariant. In this article, we observe that acting properly on median graphs is not quasi-isometry invariant either, answering Question \ref{q}. More precisely, we identify a broad source of counterexamples coming from bounded central extensions of groups acting properly (and cocompactly) on median graphs. Let us mention two specific instances of this phenomenom. The first one deals with hyperbolic groups. 

\begin{theorem}
\label{thm:main}

Let $G$ be a hyperbolic group that acts properly (and cocompactly) on a median graph. Suppose that its Schur multiplier $\HH_2(G; \mathbb{Z})$ is infinite. Then $G$ admits a central extension $G_1$ that does not act properly on a median graph, but is quasi-isometric to $G_2 := G \times \mathbb{Z}$, which acts properly (and cocompactly) on a median graph.
\end{theorem}

Examples to which this criterion applies include certain hyperbolic one-relator groups (such as surface groups) and hyperbolic von Dyck groups. 

\begin{remark}
As pointed out to us by Yves Cornulier, it is worth mentioning that \cite[Example 6.A.10]{FW} shows that the central extension of a hyperbolic surface group $\pi_1(\Sigma)$ inside $\widetilde{\mathrm{PSL}_2(\mathbb{R})}$, which is well-known to be quasi-isometric to $\pi_1(\Sigma) \times \mathbb{Z}$ (this follows from the results of Gersten and Neumann--Reeves we use in this paper, see also \cite{dastessera} for a more ad-hoc argument), does not act \emph{metrically} properly on a  median graph. In this article, an action of a group $G$ on a graph $X$ is said to be proper if it is properly discontinuous when both the group and the (vertex-set of the) graph are endowed with the discrete topology, which amounts to saying that vertex-stabilisers are finite. The action $G \curvearrowright X$ is metrically proper if $\{g \in G \mid d(x,gx) \leq R \}$ is finite for all $x \in X$ and $R \geq 0$. Thus, metrically proper actions are more restrictive than proper actions, and so the conclusion of Theorem \ref{thm:main} is stronger than that of \cite{FW} in this case.
\end{remark}

Our second source of specific examples comes from groups acting on the circle. For instance, if $\overline{T} \leq \mathrm{Homeo}(\mathbb{R})$ denotes the lift of Thompson's group $T \leq \mathrm{Homeo}(\mathbb{S}^1)$, then:

\begin{theorem}\label{thm:mainTwo}
The group $\overline{T}$ does not act properly on any median graph, and every action on a median graph with finite cubical dimension has a finite orbit. Nevertheless, $\overline{T}$ is quasi-isometric to $T \times \mathbb{Z}$, which acts properly on a median graph.
\end{theorem}

See Example~\ref{ex:ThompsonLift} for more details. The finite-orbit property satisfied by $\overline{T}$ can be proven by arguments similar to those in \cite{anthonyV}; see \cite{ToCome} for detailed proofs.  

Theorem~\ref{thm:mainTwo} does not only answer Question~\ref{q} but also provides valuable information regarding the following broad question: To which extent can two quasi-isometric groups have different median geometries? Such a question is illustrated in \cite{questions} by

\begin{question}[{\cite[Question 13]{questions}}]
\label{q2}
Does there exist a group acting properly and cocompactly on a median graph but quasi-isometric to a group satisfying Kazhdan's property (T)?
\end{question}

Indeed, because Kazhdan's property (T) can be characterised as a fixed-point property on median spaces \cite{Tmedian}, the finite-orbit property on median graph (of finite cubical dimension) can be thought of as a discrete (and finite-dimensional) version of Kazhdan's property (T). From this point of view, while Question \ref{q2} is still wide open, Theorem \ref{thm:mainTwo} hints at the fact that it might have a positive answer. Moreover, it is worth mentioning that every hyperbolic group \cite[Theorem~1.8]{MR2979855} (and more generally colourable hierarchically hyperbolic group \cite{Petyt}) is quasi-isometric to a median graph, which implies that, for instance, uniform lattices in $\mathrm{Sp}(n,1)$ satisfy Kazhdan's property (T) and are quasi-isometric to median graphs.\\

Finally, let us observe that, from our examples, it is easy to construct many new ones in the following sense:

\begin{corollary}
\label{cor:many}
\begin{enumerate}
    \item There exists an infinite family $\{ G_i \}_{i \geq 1}$ of pairwise non-isomorphic quasi-isometric groups such that $G_1$ acts properly and cocompactly on a median graph, while $G_i$ does not act properly on a median graph, for any $i \geq 2$.
    \item There exists an uncountable family $\{ G_i \}_{i \in I}$ of pairwise non-isomorphic quasi-isometric groups such that $G_{i_0}$ acts properly on a median graph, while $G_i$ does not, for any $i \neq i_0$.
\end{enumerate}
\end{corollary}

Note that a group acting properly and cocompactly on a median graph is finitely presented, and finite presentability is well-known to be quasi-isometry invariant (see e.g. \cite[Exercice~23]{MR1086650} or \cite{alonso}), therefore it is not possible to upgrade the first item of Corollary \ref{cor:many} to an uncountable family. \\

Although Question \ref{q} has a negative answer, there are several variations and generalizations thereof that could still be investigated: We discuss these in Section \ref{s:conclusion}. In particular, our counterexamples lead naturally to the definition of \emph{almost median graphs}, which seem more suitable for studying groups up to quasi-isometry. \\

\textbf{Acknowledgements.} The authors wish to thank Jason Behrstock, Yves Cornulier, Alessandro Sisto and Abdul Zalloum for useful comments.

\section{Bounded central extensions}
\label{s:obstruction}

We start by reviewing the necessary cohomological tools: See \cite[Chapter IV]{brownbook} for more details. Let $R$ be a ring; in our case we will only deal with $R = \mathbb{Z}, \mathbb{Q}$ and $\mathbb{R}$. Given a group $G$, we denote its second cohomology with coefficients in $R$ by $\HH^2(G; R)$. The inclusions $\mathbb{Z} \to \mathbb{Q} \to \mathbb{R}$ induce change of coefficient maps in cohomology. Given $\alpha \in \HH^2(G; \mathbb{Z})$, we denote by $\alpha_R \in \HH^2(G; R)$ its image under these maps, where $R = \mathbb{Q}$ or $\mathbb{R}$. Note that $\HH^2(G; \mathbb{Z}) \to \HH^2(G; \mathbb{Q})$ is in general not injective (as will be clear from Proposition \ref{prop:criterion}), however $\HH^2(G; \mathbb{Q}) \to \HH^2(G; \mathbb{R})$ is always injective: This is an immediate consequence of the Universal Coefficient Theorem \cite[Chapter I.0]{brownbook}.

Recall that $\HH^2(G; R)$ parametrizes central extensions of $G$ by $R$. More precisely, let $\alpha \in \HH^2(G; R)$, and let $\omega : G^2 \to R$ be an inhomogeneous cocycle representing it. One can then associate a central extension:
$$1 \to R \to E_\omega \to G \to 1,$$
and every central extension arises this way. Up to a suitable notion of equivalence, $E_\omega$ depends only on its class $\alpha$, therefore in the sequel we will use the notation $E_\alpha$ instead. In particular, the extension splits if and only if $\alpha = 0$.

Cohomology is functorial, in that group homomorphisms induce maps in cohomology. For our purposes, it will be sufficient to consider the restriction to a subgroup $H \leq G$, denoted $\res : \HH^2(G; R) \to \HH^2(H; R)$, which is given by restricting cocycles from $G^2$ to $H^2$. At the level of central extensions we have the following commutative diagram:
\[\begin{tikzcd}
	1 & R & {E_{\res(\alpha)}} & H & 1 \\
	1 & R & {E_{\alpha}} & G & 1
	\arrow["\id", from=1-2, to=2-2]
	\arrow[from=1-2, to=1-3]
	\arrow[from=1-1, to=1-2]
	\arrow[from=1-3, to=1-4]
	\arrow[from=1-4, to=1-5]
	\arrow[from=2-1, to=2-2]
	\arrow[from=2-2, to=2-3]
	\arrow[from=2-3, to=2-4]
	\arrow[from=2-4, to=2-5]
	\arrow[hook, from=1-3, to=2-3]
	\arrow[hook, from=1-4, to=2-4]
\end{tikzcd}\]
That is, $E_{\res(\alpha)}$ is simply the preimage of $H$ under the quotient $E_\alpha \to G$. \\

A class $\alpha \in \HH^2(G; R)$ is said to be \emph{bounded} if it is admits a cocycle representative whose image is a bounded subset of $R$, with respect to the Euclidean norm. The relevance of bounded classes to our purposes is due to the following:

\begin{proposition}[Gersten \cite{gersten}]
Let $G$ be a finitely generated group, and $\alpha \in \HH^2(G; \mathbb{Z})$ be a bounded class. Then $E_\alpha$ is quasi-isometric to $G \times \mathbb{Z}$.
\end{proposition}

The converse does not hold: The first counterexample was found by Frigerio and Sisto \cite{frigeriosisto}, and more recently Ascari and Milizia provided a finitely presented one \cite{ascarimilizia}. Let us point out that $\alpha \in \HH^2(G; \mathbb{Z})$ is bounded if and only if $\alpha_\mathbb{R}$ is bounded \cite[Lemma 2.8]{frigeriosisto}. This is especially useful given all that is known about bounded cohomology with real coefficients \cite{frigerio}.

One classical context in which bounded classes arise is via actions on the circle \cite{ghys}, which we will exploit in Subsection \ref{circle}. Another easy way to tell that a class is bounded is to apply the following:

\begin{theorem}[Neumann--Reeves \cite{surjectivity}]
\label{thm:surjectivity}

Let $G$ be a hyperbolic group. Then every class $\alpha \in \HH^2(G; \mathbb{Z})$ is bounded.
\end{theorem}

This has been vastly generalized by Mineyev \cite{mineyev1}, who also proved that the more general statement characterizes hyperbolic groups \cite{mineyev2}. \\

Clearly every (proper, cocompact) action of a group $G$ on a median graph defines a (proper, cocompact) product action of $G \times \mathbb{Z}$ on a median graph. Our goal is to show that, under certain conditions on $\alpha \in \HH^2(G; \mathbb{Z})$, the central extension $E_\alpha$ does not act properly on a median graph. We will use the following property of such groups, well-known for groups acting on CAT(0) spaces by semi-simple isometries \cite[Theorem~II.6.12]{bridsonhaefliger} and extended to arbitrary median graphs in \cite[Theorem~5.4]{CFTT} (see also \cite[Propositions 6.A.9 and 6.B.8]{FW}):

\begin{proposition}[\cite{CFTT}]
\label{prop:virtualsplitting}

Let $G$ be a group acting on a median graph $X$ and $A \lhd G$ a normal finitely generated abelian subgroup. If every non-trivial element in $A$ has unbounded orbits in $X$, then $A$ is a direct factor in a finite-index subgroup of $G$.
\end{proposition}

Our final goal is therefore to provide a cohomological criterion that prevents this behaviour.

\begin{proposition}
\label{prop:criterion}

Let $G$ be a group, and let $\alpha \in \HH^2(G; \mathbb{Z})$. Suppose that $\alpha_\mathbb{Q} \in \HH^2(G; \mathbb{Q})$ is non-zero. Then the central extension $E_\alpha$ does not virtually split, and in particular it does not act properly on a median graph.
\end{proposition}

This result is standard and well-known, in fact the converse holds as well, assuming that $G$ is finitely generated \cite[Lemma 5.13, Lemma 5.14]{frigeriosisto}. We include a proof for the reader's convenience.

\begin{proof}
Suppose that the extension virtually splits: There exists a subgroup $H$ of $E_\alpha$ that intersects $\mathbb{Z}$ trivially, and such that $H \times \mathbb{Z}$ has finite index in $E_\alpha$. Since $H$ intersects $\mathbb{Z}$ trivially, the quotient $E_\alpha \to G$ realizes $H$ as a finite-index subgroup of $G$. Now the hypothesis on $H$ means that $E_{\res(\alpha)}$ splits, and thus $\res(\alpha) \in \HH^2(H; \mathbb{Z})$ is the zero class. Consider the following commutative diagram:
\[\begin{tikzcd}
	{\HH^2(G; \mathbb{Z})} & {\HH^2(G; \mathbb{Q})} \\
	{\HH^2(H; \mathbb{Z})} & {\HH^2(H; \mathbb{Q})}
	\arrow[from=1-1, to=1-2]
	\arrow[from=1-1, to=2-1]
	\arrow[from=1-2, to=2-2]
	\arrow[from=2-1, to=2-2]
\end{tikzcd}\]
where the horizontal arrows are change of coefficients maps and the vertical arrows are restriction maps. Since $\res(\alpha) = 0$, we obtain $\res(\alpha)_\mathbb{Q} = \res(\alpha_\mathbb{Q}) = 0$. But every positive integer is divisible in $\mathbb{Q}$, so the restriction to a finite-index subgroup in rational cohomology is injective \cite[Proposition III.10.4]{brownbook}. In particular $\alpha_\mathbb{Q} = 0$.

The last statement follows from Proposition \ref{prop:virtualsplitting}.
\end{proof}

Here is an easy way to tell when there exists such a class:

\begin{corollary}
\label{cor:criterion}

Let $G$ be a group, and suppose that $\HH_2(G; \mathbb{Z})$ is an infinite finitely generated group. Then there exists $\alpha \in \HH^2(G; \mathbb{Z})$ such that $\alpha_{\mathbb{Q}} \in \HH^2(G; \mathbb{Q})$ is non-zero.
\end{corollary}

Note that the hypothesis of finite generation on $\HH_2(G; \mathbb{Z})$ is very mild: For instance it is satisfied by all finitely presented groups.

\begin{proof}
The Universal Coefficient Theorem \cite[Chapter I.0]{brownbook} gives the following diagram of short exact sequences:
\[\begin{tikzcd}
	0 & {\Ext^1_{\mathbb{Z}}(\HH_1(G; \mathbb{Z}); \mathbb{Z})} & {\HH^2(G;\mathbb{Z})} & {\Hom_{\mathbb{Z}}(\HH_2(G; \mathbb{Z}), \mathbb{Z})} & 0 \\
	0 & {\Ext^1_{\mathbb{Z}}(\HH_1(G; \mathbb{Z}); \mathbb{Q})} & {\HH^2(G;\mathbb{Q})} & {\Hom_{\mathbb{Z}}(\HH_2(G; \mathbb{Z}), \mathbb{Q})} & 0
	\arrow[from=1-1, to=1-2]
	\arrow[from=1-2, to=1-3]
	\arrow[from=1-3, to=1-4]
	\arrow[from=1-4, to=1-5]
	\arrow[from=2-1, to=2-2]
	\arrow[from=2-2, to=2-3]
	\arrow[from=2-3, to=2-4]
	\arrow[from=2-4, to=2-5]
	\arrow[from=1-2, to=2-2]
	\arrow[from=1-3, to=2-3]
	\arrow[from=1-4, to=2-4]
\end{tikzcd}\]
Since $\HH_2(G; \mathbb{Z})$ is finitely generated and infinite, there exists a non-zero homomorphism from $\HH_2(G; \mathbb{Z})$ to $\mathbb{Z}$, which defines a non-zero homomorphism from $\HH_2(G; \mathbb{Z})$ to $\mathbb{Q}$. Choosing a class $\alpha \in \HH^2(G; \mathbb{Z})$ that maps to this homomorphism, we obtain $\alpha_{\mathbb{Q}} \neq 0$.
\end{proof}

We can now deduce Theorem \ref{thm:main}:

\begin{proof}[Proof of Theorem \ref{thm:main}]
Let $G$ be a hyperbolic group that acts properly (and cocompactly) on a median graph. Then the same is true of $G_2 := G \times \mathbb{Z}$. Now $G$ is finitely presented, so $\HH_2(G; \mathbb{Z})$ is a finitely generated group. Assuming moreover that it is infinite, Corollary \ref{cor:criterion} produces a class $\alpha \in \HH^2(G; \mathbb{Z})$ such that $\alpha_\mathbb{Q} \neq 0$. Proposition \ref{prop:criterion} then shows that the corresponding central extension $G_1$ does not act properly on a median graph. But $G$ is hyperbolic, so $\alpha$ is bounded by Theorem \ref{thm:surjectivity}, and so $G_1$ is quasi-isometric to $G_2$.
\end{proof}

\section{Examples}
\label{s:examples}

We will now go through a list of examples of negative answers to Question \ref{q}. These will come in several classes: Hyperbolic one-relator groups and hyperbolic von Dyck groups, to which Theorem \ref{thm:main} will directly apply, and groups acting on the circle, where the boundedness of the relevant class is independent of hyperbolicity.

\subsection{One-relator groups}

Let $G_r = \langle S \mid r \rangle$ be a one-relator group, where $r \in [F(S), F(S)]$. The latter hypothesis is equivalent to the fact that the quotient $F(S) \to G_r$ induces an isomorphism in the abelianization. Under this assumption we have $\HH_2(G_r; \mathbb{Z}) \cong \mathbb{Z}$ (see e.g. \cite[Section 3.1]{sv:onerelator}). Thus Theorem \ref{thm:main} implies:

\begin{corollary}
Let $r \in [F(S), F(S)]$, and let $G_r$ be the corresponding one-relator group. Suppose that $G_r$ is hyperbolic and acts properly (and cocompactly) on a median graph. Then $G_r$ admits a central extension that does not act properly on a median graph, but is quasi-isometric to $G_r \times \mathbb{Z}$, which acts properly (and cocompactly) on a median graph.
\end{corollary}

\begin{example}
\label{ex:surface}

Let $\Sigma$ be a closed surface of genus $g \geq 2$. Then $\pi_1(\Sigma) = \langle a_i, b_i : i = 1, \ldots, g \mid [a_1, b_1] \cdots [a_g, b_g] = 1 \rangle$ satisfies the above hypotheses. Therefore there exists a central extension of $\pi_1(\Sigma)$ that does not act properly on a median graph, but which is quasi-isometric to $\pi_1(\Sigma) \times \mathbb{Z}$. Moreover, $\pi_1(\Sigma)$ is cocompactly cubulated: $\Sigma$ classically admits a non-positively curved square tessellation which lifts in the hyperbolic plane as a square complex whose one-skeleton is a median graph. (More precisely, think of $\Sigma$ as obtained by identifying the opposite sides of a $2g$-gon $P$, add a vertex at the center of $P$, and connect it with edges to half of the vertices of $P$ alternatively.)

This central extension has already been studied from this point of view in \cite{dastessera}. Therein, the reader can also find an elementary argument as to why this central extension is quasi-isometrically trivial, which makes no reference to (bounded) cohomology or hyperbolicity.
\end{example}

\begin{example}
Let $r \in [F(S), F(S)]$ be a proper power. The latter hypothesis is equivalent to the fact that $G_r$ has torsion. Then $G_r$ is hyperbolic and cocompactly cubulable \cite{lauerwise}. Therefore $G_r$ admits a central extension that does not act properly on a median graph, but which is quasi-isometric to $G_r \times \mathbb{Z}$, which is cocompactly cubulated.
\end{example}

\begin{example}
Let $r \in [F(S), F(S)]$ be such that the corresponding group $G_r$ has \emph{negative immersions}, as defined in \cite{negativeimmersions}. This is equivalent to the fact that $r$ has primitivity rank greater than $2$ (see \cite{negativeimmersions} or \cite{primitivityrank} for the relevant definition) or that every $2$-generated subgroup of $G_r$ is free. Then $G_r$ is hyperbolic and virtually special \cite{linton} so it follows from \cite[Theorem~3.1]{sageev} that it is cocompactly cubulated. Therefore $G_r$ admits a central extension that does not act properly on a median graph, but is quasi-isometric to $G_r \times \mathbb{Z}$, which is cocompactly cubulated.
\end{example}

\subsection{Hyperbolic von Dyck groups}

Let $a, b, c$ be positive integers such that $1/a + 1/b + 1/c < 1$. The group
$$D := D(a, b, c) := \langle x, y, z \mid x^a = y^b = z^c = xyz = 1 \rangle$$
is called a \emph{hyperbolic von Dyck group}. It is an index-$2$ subgroup in the Coxeter group
$$\Delta := \Delta(a, b, c) := \langle s, t, u \mid s^2 = t^2 = u^2 = (st)^a = (tu)^b = (us)^c = 1 \rangle.$$
which is called a \emph{hyperbolic triangle group}. Hyperbolic Coxeter groups such as $\Delta$ are cocompactly cubulated \cite{nibloreeves}, therefore this is also true of $D$. Moreover, the natural action of $D$ on the hyperbolic plane by rotations can be used to show that $\HH_2(D; \mathbb{Z}) \cong \mathbb{Z}$ (see e.g. \cite[Section 3]{triangle}). Thus Theorem \ref{thm:main} implies:

\begin{corollary}
Let $a, b, c$ be positive integers such that $1/a + 1/b + 1/c < 1$. Then $D := D(a, b, c)$ admits a central extension that does not act properly on a median graph. Moreover, this central extension is quasi-isometric to $D \times \mathbb{Z}$, which is cocompactly cubulated.
\end{corollary}

\begin{example}
As an explicit example, the $(2, 3, 7)$-von-Dyck group satisfies the previous corollary. This group and its central extension $\langle x, y, z \mid x^2 = y^3 = z^7 = xyz \rangle$, sometimes called the \emph{$(2, 3, 7)$-homology sphere group}, have often served as examples of peculiar behaviour, particularly in the theory of left-orderable groups (see e.g. \cite{thurston, manntriestino}), which leads naturally to the following subsection.
\end{example}

\subsection{Groups acting on the circle}
\label{circle}

In this subsection we go through a different class of examples, where the boundedness of the classes comes from their dynamical nature and not from hyperbolicity. Let $G$ be a group acting faithfully by orientation-preserving homeomorphisms on the circle. To this action, one associates a cohomology class $\alpha \in \HH^2(G; \mathbb{Z})$, called the \emph{Euler class}, which is moreover bounded. The corresponding central extension $\overline{G}$ is the group of homeomorphisms of the line that commute with integer translations and induce $G$ on the quotient; in particular $\overline{G}$ is left-orderable and thus torsion-free. See \cite{ghys} or \cite[Chapter 10]{frigerio} for more details. \\

While the vanishing or non-vanishing of the Euler class in bounded cohomology is well-understood from a dynamical point of view, the same cannot be said of the Euler class in cohomology, which vanishes much more often (see e.g. \cite[Section 6.2]{ghys}). On the other hand, we can use a different criterion than Proposition \ref{prop:criterion} to ensure that the corresponding central extension does not virtually split:

\begin{lemma}
Let $G$ be a group that is not virtually torsion-free, and let $E$ be a torsion-free central extension. Then the extension does not virtually split.
\end{lemma}

\begin{proof}
If the extension does virtually split, then $E$ contains a finite-index subgroup of the form $\mathbb{Z} \times H$, where $H$ is isomorphic to a finite-index subgroup of $G$. Since $E$ is torsion-free, so is $H$, and thus $G$ is virtually torsion-free.
\end{proof}

\begin{corollary}
Let $G$ be a finitely generated group of orientation-preserving homeomorphism of the circle, and let $\overline{G}$ be the corresponding central extension. Then $\overline{G}$ is quasi-isometric to $G \times \mathbb{Z}$. If moreover $G$ is not virtually torsion-free, then $\overline{G}$ does not act properly on a median graph.
\end{corollary}

\begin{example}[Theorem \ref{thm:mainTwo}]
\label{ex:ThompsonLift}
Let $T$ be Thompson's group acting on the circle \cite{thompson}. Then $T$ is simple and has torsion, so its lift to the real line $\overline{T}$ does not act properly on a median graph. (As an alternative argument, one can observe that $\overline{T}$ contains $\mathbb{Q}$ as a subgroup \cite{QinTbar}, which prevents it from acting properly on a median graph \cite{Haglund}.) However $T$ acts properly on a median graph \cite{farley} and therefore so does $T \times \mathbb{Z}$.
\end{example}

\begin{example}
More generally, let $T_{n, r}$ be a Stein--Thompson group acting on the circle \cite{bieristrebel, brown, stein}. The subgroup $T_n \leq T_{n, r}$ is simple and has torsion, therefore $T_{n, r}$ is not virtually torsion-free, so again $\overline{T_{n,r}}$ does not act properly on a median graph. However $T_{n, r}$ acts properly on a median graph \cite{farley} and therefore so does $T_{n, r} \times \mathbb{Z}$.
\end{example}

\subsection{Uncountably many examples}

We end by showing Corollary \ref{cor:many}. This will be an easy application of our previous examples and the following result:

\begin{lemma}
Let $G$ be a cocompactly cubulated group with a central extension $G_2$ quasi-isometric to $G_1 \cong G \times \mathbb{Z}$ that does not virtually split. Then for every (cocompactly) cubulated group $H$, the group $G_1 \times H$ is (cocompactly) cubulated, while $G_2 \times H$ does not act properly on a median graph.
\end{lemma}

\begin{proof}
Since $G_1$ and $H$ both act properly (and cocompactly) on a median graph, the same is true of their product. On the other hand, $G_2 \times H$ still does not have a finite-index subgroup which contains the center as a direct factor, and so it does not act properly on a median graph by Proposition \ref{prop:virtualsplitting}.
\end{proof}

\begin{proof}[Proof of Corollary \ref{cor:many}]
Let $\{ H_i \}_{i \geq 1}$ be an infinite family of pairwise non-isomorphic quasi-isometric groups acting properly and cocompactly on a median graph, for instance all the free groups of finite rank or all the fundamental groups of closed hyperbolic surfaces. Let $G_1, G_2$ be groups as in Theorem \ref{thm:main}, for instance the two central extensions of a surface group as in Example \ref{ex:surface}. Then $G_1 \times H_i$ are all cocompactly cubulated, while none of $G_2 \times H_i$ act properly on a median graph. Moreover, the $G_2 \times H_i$ are pairwise non-isomorphic as they have different abelianizations.

The second item follows from the same argument, by letting $\{ H_i \}_{i \in I}$ be an uncountable family of pairwise non-isomorphic quasi-isometric groups acting properly on a median graph. Since a finitely generated group has only countably many finitely generated subgroups, it follows that among the groups $G_2 \times H_i$ there are still uncountably many isomorphism classes. For instance, one can consider pairwise distinct finite central extensions of the Grigorchuk group or of the lamplighter group $\mathbb{Z} \wr \mathbb{Z}$ as shown in \cite{erschler}, these two groups acting properly on median graphs according to \cite{actionswreath} (see also \cite{hilbertian}) and \cite{GrigorCube}.

Let us sketch a direct elementary construction (in the same spirit) of distinct finite central extensions of the lamplighter group $L_2:=\mathbb{Z}/2\mathbb{Z} \wr \mathbb{Z}$. This group admits 
$$\langle a,t \mid a^2=1, [t^nat^{-n},a]=1 \ (n \in \mathbb{N}) \rangle$$
as a presentation. Now, fix an arbitrary subset $I \subset \mathbb{N}$ and define the new group
$$G_I:= \left\langle a,t,z \mid a^2=z^2=[z,a]=[z,t]=1, [t^nat^{-n},a] = \left\{ \begin{array}{cl} 1 & \text{if } n \in I \\ z & \text{if } n \notin I \end{array} \right. \right\rangle.$$
In other words, we add a central element of order two and we use it to ``twist'' the commutator relations of $L_2$. We get a central extension
$$1 \to \mathbb{Z}/2\mathbb{Z} \to G_I \to L_2 \to 1.$$
The claim is that, for all $I,J \subset \mathbb{N}$, $G_I$ and $G_J$ are isomorphic if and only if $I=J$. The key observation is that every automorphism of $L_2$ (as described, for instance, in \cite{matuccisilva} or \cite{WreathMorphisms}) extends to $G_I$. Thus, if there exists an isomorphism $G_I \to G_J$, then we can assume without loss of generality that it induces the identity when we mod out by the centers. Such an isomorphism must send the element $[t^nat^{-n},a]$ of $G_I$ to the same element $[t^nat^{-n},a]$ of $G_J$ for every $n \geq 1$, which implies that $I=J$. 
\end{proof}

\section{Concluding remarks}
\label{s:conclusion}

In this note, we have proved that many central extensions do not act properly on median graphs, the geometric obstruction being recorded by Proposition~\ref{prop:virtualsplitting}. Such a virtual splitting is well-known for loxodromic isometries in CAT(0) spaces, and it is reasonable to expect the same property to hold in various other nonpositively curved spaces. 

In that spirit, define a metric space $X$ to be \emph{loxodromically indicable} if, for every \emph{loxodromic isometry} $g \in \mathrm{Isom}(X)$ (i.e. $n \mapsto g^n \cdot x$ induces a quasi-isometric embedding $\mathbb{Z} \to X$ for every $x \in X$), there exists a morphism $\varphi : C(g) \to \mathbb{Z}$ satisfying $\varphi(g) \neq 0$, where $C(g)$ denotes the centralizer of $g$ in $\mathrm{Isom}(X)$. Median graphs and CAT(0) spaces are loxodromically indicable (see the proof of \cite[Theorem~II.6.12]{bridsonhaefliger} for CAT(0) spaces; the proof of \cite[Theorem~5.4]{CFTT} and the improvement \cite[Proposition~11.2.1]{MedianBook} for median graphs), and a virtual splitting similar to Proposition~\ref{prop:virtualsplitting} holds for every loxodromically indicable space. More precisely, it follows from the proof of \cite[Theorem~II.6.12]{bridsonhaefliger}:

\begin{proposition}
Let $G$ be a group acting  on a loxodromically indicable space $X$. If $A \leq G$ is a central finitely generated subgroup all of whose non-trivial elements are loxodromic, then $A$ is a direct factor in a finite-index subgroup of $G$.
\end{proposition}

Investigating for more examples of loxodromically indicable spaces is an interesting problem. Metric spaces for which Axis and Flat Torus Theorems are known provide good candidates, such as systolic complexes, Helly graphs, spaces with good bicombings, Garside groups. \\

In another direction, one could claim that the central extensions constructed in this note are not so different from groups acting on median graphs and that the notion of median graphs is too rigid in order to include them. A key observation is that our central extensions do not admit morphisms to $\mathbb{Z}$ that are non-trivial on their centers but they do admit quasi-morphisms that are unbounded on their centers. This motivates the idea to relax the definition of median graphs in order to encompass quasi-lines.

\begin{definition}
Let $X$ be a graph. Given two vertices $x,y \in X$ and some constant $\delta \geq 0$, define the \emph{$\delta$-interval} 
$$I_\delta(x,y):= \{ z \in X \mid |d(x,y)-d(x,z)-d(z,y)| \leq \delta\}.$$
A graph is \emph{almost median} if there exist $\delta,\Delta \geq0$ such that, for all $x,y,z \in X$, the intersection
$$I_\delta(x,y) \cap I_\delta(y,z) \cap I_\delta(x,z)$$
is non-empty and has diameter $\leq \Delta$. 
\end{definition}

In other words, we allow some additive error in the definition of median graphs. This general process has been applied successfully in other contexts, for instance in the study of \emph{coarse Helly graphs}.

When $\delta=\Delta=0$, we recover the definition of median graphs. As desired, every quasi-line (and, in fact, every hyperbolic graph) is almost median. See \cite{elder, DrutuChatterji} for similar definitions. We emphasize that almost median graphs are distinct from \emph{coarse median spaces} introduced in \cite{MR3037559}. Almost median graphs are reasonably coarse median, but the converse is probably false.

\begin{proposition}
Let $1 \to \mathbb{Z} \to E \to Q \to 1$ be a bounded central extension. If $Q$ acts properly on an almost median graph, then so does $E$.
\end{proposition}

\begin{proof}
We can describe $E$ as the set $\mathbb{Z} \times Q$ endowed with the product
$$(z_1,q_1) \cdot (z_2,q_2) = (z_1+z_2+c(q_1,q_2), q_1q_2), \ z_1,z_2 \in \mathbb{Z}, q_1,q_2 \in Q$$
for some inhomogeneous bounded cocycle $c$. The map $(z,q) \mapsto z$ defines a quasi-morphism $E \to \mathbb{Z}$, from which we can define an action of $E$ on a quasi-line $L$ \cite[Lemma~4.15]{hyperbolicstructures}. Therefore, if $M$ is an almost median graph on which $Q$ acts properly, then $E$ naturally acts on $L \times Q$, and the action is proper. An $\ell^1$-product of almost median graphs being almost median, the desired conclusion follows.
\end{proof}

A natural problem to investigate is to which extent the geometry of almost median graphs is similar to the geometry of median graphs. As a consequence of the above proposition, almost-median graphs are no longer loxodromically indicable, but it may be expected that a coarse version of the property still holds: If $X$ is an almost median graph and $g \in \mathrm{Isom}(X)$ a loxodromic isometry, does there exist a quasi-morphism $\mathrm{Isom}(X) \to \mathbb{R}$ that is unbounded on $\langle g \rangle$?

Following \cite{bridson}, it is shown in \cite{CFTT} that mapping class groups cannot act properly on median graphs, precisely because the latters are loxodromically indicable. Interestingly, one can reasonably say that this is the only obstruction since mapping class groups act properly on products of hyperbolic spaces \cite{BBF} (even quasi-trees \cite{MR4315766}) and a fortiori on almost median graphs. However, it is not clear whether such an action on an almost median graph can be in addition cocompact.   \\

Finally, since the groups we considered all come from central extensions, it would be interesting to produce counterexamples to Question \ref{q} among center-free groups. A reasonable source of candidates is provided by \cite{commensuratingHNN}, as it follows from \cite{huangprytula} that they cannot act properly on median graphs with finite cubical dimension.

\vspace{0.5cm}

\footnotesize 
\bibliographystyle{alpha}
\bibliography{ref}

\vspace{0.5cm}

\normalsize

\noindent{\textsc{Department of Mathematics, ETH Z\"urich, Switzerland}}

\noindent{\textit{E-mail address:} \texttt{francesco.fournier@math.ethz.ch}} \\

\noindent{\textsc{Institut Montpellierain Alexander Grothendieck, Montpellier, France}}

\noindent{\textit{E-mail address:} \texttt{anthony.genevois@umontpellier.fr}}

\end{document}